\documentclass[12pt,a4paper,twoside,english]{smfart}

\usepackage{lmodern,mdframed,xcolor}
\usepackage[utf8]{inputenc}
\usepackage[francais]{babel}

\usepackage[OT2,T1]{fontenc}

\mdfdefinestyle{MyFrame}{%
    linecolor=white,
    innerrightmargin=5pt,
    innerleftmargin=5pt,
    splittopskip=\topskip,
    backgroundcolor=white}

\usepackage{amssymb}
\usepackage{amsmath}
\usepackage{smfthm,mathabx}
\usepackage{enumerate}
\usepackage{amsthm}
\usepackage{amsfonts}
\usepackage{graphicx}
\usepackage[colorlinks=true,linkcolor=black,citecolor=black,urlcolor=black]{hyperref}
\usepackage{fancyhdr}
\usepackage[all,cmtip]{xy}
\usepackage{alltt}
\usepackage{footmisc}
\usepackage[all]{xy}
\usepackage{amscd}
\usepackage{verbatim}
\usepackage{amsmath,color}
\font\rurm=wncyr10 scaled \magstep1
\usepackage{mathtools}

\usepackage{verbatim}
\usepackage{latexsym}
\usepackage{tikz}
\usetikzlibrary{matrix,arrows}
\usepackage{hyperref}
\font\rurm=wncyr10 scaled \magstep1

\usepackage{calrsfs}
\usepackage{fullpage}

\xyoption{all}
\xyoption{frame}

\def\Q{{\mathbb Q}}
\def\Z{{\mathbb Z}}
\def\fq{{\mathbb F}}

\def\Reel{{\mathbb R}}

\def\p{{\mathfrak p}}
\def\P{{\mathfrak P}}

\def\q{{\mathfrak q}}

\def\QQ{{\mathfrak Q}}

\def\f{{\mathfrak f}}

\def\sha{{{\textnormal{\rurm{Sh}}}}}
\def\CyB{{{\textnormal{\rurm{B}}}}}

\def\N{{\rm N}}
\def\R{{\rm R}}

\def\G{{\rm G}}
\def\HH{{\rm H}}
\def\J{{\rm J}}
\def\K{{\rm K}}
\def\L{{\rm L}}
\def\F{{\rm F}}

\def\V{{\rm V}}

\def\Gg{{\rm G}}

\def\rd{{\rm rd}}

\def\O{{\mathcal O}}

\def\SS{{\mathcal S}}

\def\PP{{\mathcal P}}

\def\Rd{{\rm Rd}}

\def\Gal{{\rm Gal}}

\def\Cl{{\rm Cl}}

\def\ker{{\rm ker}}

\def\FF#1#2{{\displaystyle{\left(\frac{#1}{#2}\right)}}}

\newenvironment{Question}{\begin{enonce}{Question}}{\end{enonce}}
\newenvironment{conjecture}{\begin{enonce}{Conjecture}}{\end{enonce}}

\begin{document}

\date{\today}

\title{Cutting  towers of number fields}
\author{Farshid Hajir, Christian Maire, Ravi Ramakrishna}
\address{Department of Mathematics, University of Massachussetts, Amherst, MA 01003, USA}
 \address{FEMTO-ST Institute, 15B avenue des Montboucons, 25000 Besancon, FRANCE} 
\address{Department of Mathematics, Cornell University, Ithaca, USA}
\email{hajir@math.umass.edu, christian.maire@univ-fcomte.fr, ravi@math.cornell.edu}

\thanks{This work has been done during a visiting scholar position for the second author at Cornell University for  the academic year 2017-18, and funded by the program "Mobilit\'e sortante" of the R\'egion Bourgogne Franche-Comt\'e;  CM thanks the  Department of Mathematics at Cornell University for providing a beautiful research atmosphere.
The third author was supported by Simons Collaboration grant 524863.}


\begin{abstract} Given a prime $p$, a number field $\K$ and a finite set of places $S$ of $\K$,
let $\K_S$ be the maximal pro-$p$ extension of $\K$ unramified outside $S$. Using the Golod-Shafarevich criterion one can
often show that $\K_S/\K$ is infinite. In both the tame and wild cases we construct infinite subextensions with bounded ramification using the refined Golod-Shafarevich criterion. In
 the tame setting we achieve new records on Martinet constants (root discriminant bounds) in the totally real and totally complex cases. 
 We
are also able to answer a question of Ihara by producing infinite asymptotically good extensions in which infinitely many primes split completely.
\end{abstract}

\maketitle

\tableofcontents


Let $p$ be a prime number, 
 $\K$ a number field and $S$ a finite set of  finite places of $\K$. Denote by $\K_S$ and $\G_S$ the maximal pro-$p$ extension of $\K$ unramified outside $S$ 
and  $\Gal(\K_S/\K)$ respectively.
When $\G_S$ is infinite, we will exhibit some special infinite quotients  of $\G_S$. We establish infiniteness by using 
 an old, but not so well-known 
 refinement of the Theorem of Golod-Shafarevich.
The best uses of this refinement
are probably the work on deep relations of
Koch \cite[Chapter 12]{Koch}, Koch-Venkov \cite{Koch-Venkov}, Kisilevski-Labute \cite{Kisilevski-Labute} and Schoof \cite{Schoof} (when $S=\emptyset$).

\smallskip

More precisely, by starting with an infinite tower $\K_S/\K$, we cut (quotient) $\G_S$ in order to produce three kind of results:
\begin{enumerate}
\item[$(i)$]  We cut by  Frobenii; this allows us to produce asymptotically good extensions where the set of places
that split completely  is infinite: this was an open question of  Ihara \cite{Ihara};
\item[$(ii)$]  We cut by powers of the generators of tame inertia groups; this allows us to improve upon the results of  Hajir-Maire   \cite{Hajir-Maire-JSC} in the totally
complex case 
and Martin \cite{Martin} in the totally real case concerning Martinet constants 
(root discriminants) 
in infinite towers of number fields (also see \cite{Martinet}). Our improvements towards the GRH lower bounds in the totally complex and totally real cases are, by one metric, about 7.55\% and 4.36\%. \color{black}
In contrast, the improvements made
in \cite{Hajir-Maire-JSC} over \cite{Martinet} were, respectively, 16.27\% and 6.54\%;
\item[$(iii)$] In the wild ramification context, we cut by   local commutators and  powers of  generators of the inertia group. 
As an application, let $\K$ be a  totally imaginary number field of degree at least $12$ and let $\K_{S_p}$ be its maximal pro-$p$ extension
ramified only at primes above $p$.   Then, for infinitely many primes $p$ (assuming  a recent conjecture of Gras concerning $p$-rational fields \cite{Gras-CJM}), the extension $\K_{S_p}/\K$ contains an infinite unramified tower of number fields.
\end{enumerate}

In  the context of pure pro-$p$ group theory,  Wilson in fact   used the idea of cutting in \cite{Wilson}. If we translate his result to our Galois context, we  obtain  the existence of 
an infinite subextension in $\K_S/\K$ for which  all the Frobenius elements are torsion.  


\

Recall that for a number field $\K$ with $[\K:\Q]=n$, the root discriminant of $\K$, denoted $\rd_\K$, is
$|\rm{Disc}(\K)|^{1/n}$.
There are absolute lower bounds, improved over the years, that
include terms that go to $0$ as $n \to \infty$. These lower bounds depend on the signature of $\K$ and 
have been achieved by analytic methods. The best lower bounds depend on the GRH. 

\

The term that goes to $0$ with increasing degree makes it natural to consider towers of number fields and take the $\limsup$
of the root discriminants. For the $p$-power cyclotomic tower it is an exercise to see this $\limsup$ is $\infty$. It is also
an exercise to see that root discriminants are constant in unramified extensions. Thus the work of Golod and Shafarevich establishing the
existence of infinite Hilbert Class Field towers also immediately gave a rich supply infinite towers with bounded root discriminants. 
Recall Euler's constant $\gamma:= \lim_{n\to\infty}\left( (\sum^n_{k=1} \frac1k) -\log n \right)$.
The current GRH lower bounds for infinite towers are
$8\pi e^{\gamma} \approx 44.763$ for totally complex fields and $8\pi e^{\gamma+\frac{\pi}2} \approx 215.33$  for totally real fields.
See \cite{Odlyzko} for a nice history of this work up until 1990.

\

It is also natural to seek explicit examples of infinite towers with small $\limsup$ of the root discriminants. As mentioned above, Martinet and then Hajir-Maire
gave totally real and totally complex infinite towers with small root discriminant. Hajir-Maire introduced the idea of allowing tame ramification. One can show
the relevant Galois groups are infinite using the Golod-Sharafevich criterion, and the root discriminants can be bounded by tame ramification theory. 
In this paper we establish that our technique 
of cutting works well in both the tame and wild cases.

\

\

\

{\bf Notations}:

We fix a rational prime  $p$. 

$-$ 
Given a $\Z$-module $M$, we denote by $d_p M$ the dimension over $\fq_p$ of $\fq_p \otimes M$: it is the $p$-rank of $M$. 

\medskip

$-$ We fix  a number field $\K$. By abuse, we identify prime ideals $\p \subset \O_\K$ and places~$v$ of~$\K$.  For a place $v$ of $\K$, we denote by $\K_v$ the completion,  and by $U_v$ the local units.  For  $v$ is finite, we denote by 
 $\pi_v$ an uniformizer and  by $v$ the corresponding valuation.

For a prime ideal $\p\subset \O_\K$, we put $\N(\p):=\# (\O_\K/\p)$ its absolute norm.
\color{black} 
\medskip

$-$ 
We fix now a finite set $S$ of finite places of $\K$.
For $\p \in S$, not dividing $p$, one assumes that $\N(\p) \equiv 1 \,{\rm mod} \ p$. 
When all $\p \in S$ are prime to $p$ we call $S$ {\it tame}. 
The set of all $\p \subset \O_\K $ dividing $(p)$ is denoted $S_p$.

If $\L/\K$ is a finite extension, we also denote by $S$ the set of places of $\L$ above~$S$.

\medskip

$-$ 
The maximal pro-$p$ extension of $\K$ unramified outside $S$ is denoted $\K_S$. Note $\K_\emptyset$ is the maximal everywhere unramified pro-$p$ extension of $\K$.
The Galois group of $\K_S/\K$
is denoted $\G_S$. 

\medskip

$-$
All cohomology groups of $\G_S$ have coefficients in $\Z/p$, and
we denote by $d=d(\G_S)$ and $r=r(\G_S)$, $d_p H^i(\G_S)$ for $i=1,2$ respectively.
Hence, $d(\G_S)=d_p (\G_S/[\G_S,\G_S])$. More generally for a finitely generated pro-$p$ group $\G$, we denote by $d_p \G:= d(\G)$ its $p$-rank.

\medskip

$-$ 
Set 
$\delta_{\K,p}= \left\{ \begin{array}{rr} 1 & \mu_p \subset \K  \\
0 & \rm{otherwise} \end{array} \right.$.
If $\K$ has signature $(r_1,r_2)$, set $\alpha_{\K,S}:=2+2 \sqrt{r_1+r_2+\theta_{\K,S}}$
where $\theta_{\K,S}=\left\{ \begin{array}{ll} 0 & S \neq \emptyset  \\
\delta_{\K,p} & S=\emptyset \end{array} \right.$.
It is well-known that  when $S$ is tame: $d(\G_S) > \alpha_{\K,S} \Longrightarrow r <  d^2/4$. More precisely, if $P(t)=1-dt+rt^2$ then  $P(t_0)<0$ for $\displaystyle{t_0=d/2r \in  ]0,\frac{1}{2}]}$.

\medskip

$-$ 
For a number field $\L$ the root discriminant of $\L$, denoted $\rd_{\L}$, is  $|{\rm Disc}(\L)|^{1/[\L:\Q]}$. If $\J$ is an infinite algebraic extension of $\Q$,
we set $\rd_{\J}:= \limsup_\L |{\rm Disc}(\L)|^{1/[\L:\Q]}$ where the limit is taken over all number fields $\L \subset \J$. In the relative setting, e.g $\L/\K$,
we compute discriminants to $\K$ and use the degree over $\K$ when taking roots.
\color{black}

\medskip


\section{Depth of relations and the Theorem of Golod-Shafarevich} Good references are \cite[Appendice]{Lazard} and \cite{Koch}.

\subsection{The Zassenhaus filtration}
Consider a finitely generated pro-$p$ group~$\G$. 
Denote by $I:=I_\G$ the augmentation ideal of the completed
algebra $\fq_p\ldbrack \G \rdbrack$, {\emph i.e.}, $I_\G:= \ker(\fq_p \ldbrack \G \rdbrack \rightarrow \fq_p)$. The powers $I^n$ of $I$ are closed ideals and 
topologically
 finitely generated over $\fq_p$.

\begin{defi} Given  $x\in \G$, $x\neq 1$, denote by $\omega(x):=\max\{n, \ x-1 \in I^n\}$. We call the integer $\omega(x)$  the {\it depth} or
{\it level/weight}  of $x$. Put $\omega(1)=\infty$.
\end{defi}

Recall the Zassenhaus filtration of $\G$:
$$ \G_n=\{g\in \G, \omega(g) \geq n \}, \ n \geq 1$$ and
 that $I/I^2 \simeq \G/[\G,\G]\G^p$, hence, $\G_2=[\G,\G]\G^p$.
The sequence $\G_n$ is a filtration of open normal subgroups of $\G$.

\begin{prop} \label{prop:proprietes-omega} One has the  following properties:
\begin{enumerate}
\item[$(i)$] For $x \in \G$,  $\omega(x^p) \geq p \omega(x)$. Hence, if $x\in \G_n$,  then $x^p \in \G_{np}$;
\item[$(ii)$] For $x\in \G_n$ and $y\in \G_m$, one has  $[x,y] \in \G_{n+m}$;
\item[$(iii)$] For $x,y \in \G$, $\omega(xy) \geq \min (\omega(x),\omega(y))$.
\end{enumerate}
\end{prop}

\begin{proof} 
See \S 7.4 of \cite{Koch}.
\end{proof}

Hence $\omega$ is a restricted filtration following the terminology of Lazard (see \cite[Appendice A2, Definition 2.2 page 201]{Lazard}). In fact, one even has:
\begin{prop}
The Zassenhaus filtration of a pro-$p$ group $\G$,   is the minimal   filtration on $\G$ satisfying:
\begin{enumerate}
\item[$(i)$] $\omega(\G)\geq 1$; 
\item[$(ii)$] $\omega(x^p) \geq p \omega(x)$; 
\item[$(iii)$] for all $\nu>0$, $\G_\nu:=\{x\in \G, \omega(x)\geq \nu\}$ are closed. 
\end{enumerate}
\end{prop}

\begin{proof} See Lazard \cite[Appendice A2, Theorem 3.5 page 205]{Lazard}.
\end{proof}

\subsection{The inequality}

Let $\G$ be a finitely generated pro-$p$ group of $p$-rank $d$. Let 
$$1\longrightarrow \R \longrightarrow \F \stackrel{\varphi}{\longrightarrow} \G \longrightarrow 1$$ be a minimal presentation $(\PP)$ of $\G$: here $\F$ is a free pro-$p$ group on $d$ generators with its  Zassenhaus filtration $\omega$.

Suppose that $\R$ can be generated by $\{ \rho_i, i=1,\cdots \}$ as a normal subgroup of $\F$. We denote $\R=\langle  \rho_i, i=1,\cdots \rangle^{\rm N}$. For $k\geq 1$, put 
 $$r_k= \#\{ \rho_i, \ \omega(\rho_i)=k\}$$
and  assume each $r_k $ is finite.
Usually one assumes that $r_k=0$ for large $k$, but this is not necessary and we will not do so in Theorem~\ref{theo:ihara}.
 We denote by $$\displaystyle{P_\PP(t):=\sum_{k\geq 2} r_k t^k -dt+1  \in \Reel\ldbrack t \rdbrack }$$  the Golod-Shafarevich polynomial (in fact series)  associated to this presentation.
  If we have no information about the depth of the relations of $\R$, then we take $P_\PP(t)=1-dt+rt^2$, where $r=d_p H^2(\G)$ (when it is finite).

 \begin{theo}[Vinberg \cite{Vinberg}] \label{theo:GS}
 If $\G$ is finite, then $$P_\PP(t) >0, \ \forall t \in ]0,1[.$$
 \end{theo}
  
  \begin{proof}
Adapt the proof of \cite{Koch}, or see \cite{Andozski}.  See also Anick \cite{Anick}. 
  \end{proof}

 \begin{rema}
One may have partial information on the depth of the relations. For example, assume that  one has only $\omega(\rho_k) \geq a_k$ for all $k$.
Then as $1-dt+\sum_{\rho_k} t^{a_k} \geq P_\PP(t)$ (coefficient by coefficient), if $1-dt+\sum_{\rho_k} t^{a_k}$ has a root on $]0,1[$, then $\G$ is infinite. 
When one has to assume all relations  are  depth two,
we obtain the Golod-Shafarevich inequality `$r \leq d^2/4$' that guarantees  $\G$ is infinite.
We say that $\G$ {\it passes the Golod-Shafarevich test} if there exists $t_0 \in ]0,1[$ such that  $P_\PP(t_0) <0$. Observe that if $P_\PP(t_0)=0$ then $\G$ is infinite, but for many of our  applications we need some $t_0$ such that $P_\PP(t_0)<0$. 
Finally, it is known that if $\G$ passes the G-S test and $d\geq 2$, then $\G$ is {\it not} analytic.
  \end{rema}

\subsection{Detecting the depth of an element}

We start with a minimal presentation $(\PP)$ of $\G$:
$$1\longrightarrow \R \longrightarrow \F \stackrel{\varphi}{\longrightarrow} \G \longrightarrow 1.$$

One has the following result that gives a relation between the Zassenhaus filtrations $\omega_\F$ of $\F$ and $\omega_\G$ of   $\G$.

\begin{prop} \label{lemma:Zassenhausfiltration}
The Zassenhaus filtration $\omega_\G$ of $\G$ coincides with the  filtration quotient of the Zassenhaus filtration $\omega_\F$ of $\F$:
for all $g\in \G$,  $\omega_\G(g)=\max\{\omega_\F(x), \ \varphi(x)=g\}$.
Hence, for $i\geq 1$, $$\G_i \simeq \F_i \R/\R \simeq \F_i/\F_i\cap \R.$$
\end{prop}

\begin{proof}
See Lazard \cite[Appendice 3, Th. 3.5, page 205]{Lazard}.
\end{proof}

\begin{defi}\label{defi:Frattini} The Frattini series of a pro-$p$ group $\G$ is defined by
$\Phi_1(\G)=G$ and $\Phi_{i+1}(\G) = \G^p_i[\G_i,\G_i]$.

\end{defi}

Given $g\in \G$ and $x\in \F$ such that $\varphi(x)=g$. Often, one wants to detect the depth of $x\in \F$.

\begin{prop}\label{corollary:detect} Put $\varphi(x)=g\in \G$ such that $\omega_\G(g)=\omega_\F(x)$.
One has $\omega_\F(x)\geq  k$ if and only if,  $g \in \G_k$. 
In particular, 
\begin{enumerate}
\item[$(i)$]  $g \in  \G^p[\G,\G] \implies \omega_\F(x) \geq 2$;
\item[$(ii)$]   $g  \in \G_2^p[\G,\G_2] \implies \omega_\F(x) \geq 3$;
\item[$(iii)$]  $g \in \Phi_n(\G) \implies  \omega_\F(x) \geq 2^{n-1}$. 
\end{enumerate}
\end{prop}

\begin{proof}
$(i)$ Recall that $\G_2=\G^p[\G,\G]$.

$(ii)$ Consider  the $p$-central descending series $\G_{(n)}$.
Then $\G_{(n)} \subset \G_n$ for all $n\geq 1$ (see for example \S 7 of \cite{Koch}).   Hence, if $g \in \G_{(3)}=\G_2^p[\G,\G_2]$, then $g\in \G_3$ which is equivalent to  $\omega_\G(g) \geq 3$, and then $\omega_\F(x) \geq 3$.

$(iii)$   Clearly $\Phi_2(\G)= \G_2$ and as $\Phi_{n+1}(\G) = ( \Phi_n(\G))^p [ \Phi_n(\G),\Phi_n(\G)] \subset \G_{2^n}$. 
\end{proof}

  \subsection{On some special quotients of a pro-$p$ group $\G$}
We study some {\em infinite} quotients of $\G$.

\begin{prop} \label{prop-seriesofquotient}
Take $(x_i)_{i \in I} \in \G$  a family of elements of $\G$ 
with $\omega_\G(x_i) \geq 2$ for all $x_i$.
Put $\Gamma:=\G/\langle x_i, i \in I \rangle^\N$. Suppose that $\G$ is given by a minimal presentation $(\PP)$ with Golod-Shafarevich polynomial $P_\PP(t)$. Then for the natural minimal presentation $$\displaystyle{1\longrightarrow \langle \R, \langle y_i,i \in I \rangle \rangle^\N  \longrightarrow \F 
\stackrel{\varphi'}{\longrightarrow} \Gamma \longrightarrow 1}$$ of $\Gamma$, where $y_i$ are chosen such that $\varphi(y_i)=x_i$ 
and $\omega_\F(y_i)=\omega_\G(x_i)$,
 the Golod-Shafarevich polynomial  of $\Gamma$ is $P_\PP(t) + \sum_{i\in I} t^{\omega_\G(x_i)}$.
\end{prop}

\begin{proof}
Obvious.
\end{proof}  
  
Much of what we need follows from this easy proposition:

\begin{prop} \label{prop:cutting}
Let $\G$ be a finitely presented
 pro-$p$ group with a minimal presentation $(\PP)$ and 
 Golod-Shafarevich polynomial  $P_\PP(t)$ of $\G$ (following a given minimal presentation $(\PP)$).  Suppose that  $P_\PP(t_0) <0$ for  some $t_0 \in ]0,1[$. 
Then for suitable $k \geq 2$ and $k' \geq 2$ we have $P_\PP(t_0) +t^k_0 <0$ and $\displaystyle P_\PP(t_0) +\frac{t^{k'}_0}{1-t_0} <0$.

\begin{enumerate}
\item[$(i)$] Let $\f \in \F$ be such that $\omega_\F(\f) \geq k$. Then the group quotient 
$\F/\langle R,\f \rangle^\N \simeq   \G/\varphi(\langle \f \rangle^{\rm N})$ is also infinite.
\item[$(ii)$] Take a sequence $(\f_i)_i \in \F$, $i\geq 0$, such that $\omega_\F(\f_i) \geq k'+i$. Then the group quotient 
$\F/\langle \R, \f_i, i \rangle^{\rm N}  \simeq \G/\varphi(\langle \f_i, i \rangle^{\rm N})  $ is also infinite.

\end{enumerate}
\end{prop}  

\begin{proof} Let us  start with a minimal presentation $1\longrightarrow \R \longrightarrow \F \stackrel{\varphi}{\longrightarrow} \G \longrightarrow 1$ of $\G$.

$(i)$. When one adds a relation of depth at least $k\geq 2$,  the $p$-rank does not change. 
Thus $P_{\PP'}(t)=P_\PP(t) + t^k $ and $k$ is chosen so $P_{\PP'}(t_0)<0$.
Theorem \ref{theo:GS} gives that $\Gamma$ is infinite.

$(ii)$.  Put $\Gamma:= \G/\varphi(\langle \f_i, i \rangle^{\rm N})$.  First, as the elements $\f_i$ are in $\F_{k'+i}$, with $k'+i\geq 2$, then $d_p \Gamma=d_p \G=d$.
Take now  the following minimal presentation $(\PP')$ of $\Gamma$: $$1\longrightarrow \R' \longrightarrow \F \longrightarrow \G \longrightarrow 1,$$ where $\R'=\langle \rho_j, \f_i, i,j \rangle^{\rm N}$, the $\rho_j$ being some generators of $\R$ as normal subgroup of $\F$. Now, the polynomial  of Golod-Shafarevich for $\Gamma$ and $(\PP')$ can be written as (thanks to Proposition \ref{prop-seriesofquotient}):
 $$P_{\PP'}(t) \leq P_\PP(t) + t^{k'} \sum_{i\geq 0} t^{i}=P_\PP(t) +t^{k'} \frac{1}{1-t},$$ which is negative at $t=t_0$ by assumption.  
 Apply Theorem \ref{theo:GS} as in $(i)$.
 \end{proof}

Sometimes one can say a little bit more.
As usual, let $d$ and $r$ be the number of generators and relations of $\G$.
Set $a=2r/d$. Suppose $a>1$ and put $$\lambda=\left\{\begin{array}{ll} \lfloor a \rfloor  & \ \ a \notin \Z \\
 a -1 & \  \ a \in \Z_{>0} 
 \end{array}
 \right. .$$
 Choose now $m \in \Z_{\geq 2}$ such that  $$1-\frac{d^2}{4r} +  \frac{\Big(\frac{\lambda}{a}\Big)^{m}}{1-\frac{\lambda}{a}}< 0.$$

\begin{lemm} \label{prop:cutting2}
Let $\G$ be a finitely presented
 pro-$p$ group  such that $r<d^2/4$  and $2r/d>1$. Take  $\lambda$ and $m$ as above.
For $k\geq 0$, take  distincts elements $\f_{j,k} \in \F$ with $j=1,\cdots, \lambda^{k+m}$, such that $\omega_\F(\f_{j,k}) \geq m+k$. Then the group quotient $\G/\varphi(\langle \f_{j,k}, \ j,k\rangle^{\rm N})$ is also infinite.
\end{lemm}

\begin{proof} Put $\Gamma:= \G/\varphi(\langle \f_{j,k}, j,k \rangle^{\rm N})$.
As before the new relations  all have depth at least two
so $d_p \Gamma = d_p \G=d$.
 Take for $\G$ the polynomial $P_{\PP}(t)=1-dt+rt^2$. Here, the Golod-Shafarevich polynomial  $P_{\PP'}$ of $\Gamma$ can be taken as 
 $$P_{\PP'}(t)=P_{\PP}(t)+\sum_{k \geq 0}^\infty \lambda^{m+k}  t^{m+k}. $$

As $\lambda < 2r/d$,   the series converge in the neighborhood of $t_0=d/2r$.
Hence, one has: $$P_{\PP'}(d/2r) \leq 1-\frac{d^2}{4r} +  \frac{\Big(\frac{\lambda}{a}\Big)^{m}}{1-\frac{\lambda}{a}}<0,$$
 and Theorem \ref{theo:GS} applies.
\end{proof}

\color{black}


\section{Infinitely many splitting in $\K_S/\K$: Ihara's question }
\subsection{Wilson's result}
In this section we address a question of Ihara. Our main result,
Theorem~\ref{theo:ihara},
builds on a group-theoretic result of Wilson \cite{Wilson}, the number-theoretic interpretation of which we give below.

\smallskip

\begin{theo} \label{theo:wilson} Let $\K$ be a number field, and $S$ be a finite set of places of $\K$ coprime to $p$.  Suppose  that $d_p \G_S > \alpha_{\K,S}$. Then there exists an  infinite  pro-$p$ extension ${\tilde \K}/\K$ in $\K_S/\K$ where  all Frobenius elements are  torsion. 
\end{theo} 

\begin{proof}

Let $P_\PP(t)$ be the Golod-Shafarevich polynomial of $\G_S$. 
 Let $x_1,x_2,\dots$ be an enumeration of all Frobenii of primes not in $S$. We know  $P_\PP(t_0) =-\delta <0$
for some  $t_0 \in ]0,1[$. Using Proposition~\ref{prop:proprietes-omega} $(i)$ and Proposition~\ref{prop:cutting} we can add in a relation corresponding to a suitable $p$-power of $x_1$ 
so that the new Golod-Shafarevich polynomial with this relation imposed is $P_\PP(t)+t^{k_1}$ and
$P_{\PP }(t_0) + t^{k_1}_0< -\delta/2$.
Now add in  a suitable relation corresponding to a power of $x_2$ and the new Golod-Shafarevich polynomial with this relation imposed is 
$P_\PP(t)+t^{k_1} +t^{k_2}$ and
$P_{\PP }(t_0) + t^{k_1}_0 +t^{k_2}_0< -\delta/2$. Continuing on with powers of $x_3, x_4$ etc.
the resulting series, $\tilde{P}_\PP(t)$ satisfies
$\tilde{P}_\PP(t_0) \leq -\delta/2 < 0 $ so the corresponding quotient of $\G_S$, fixing ${\tilde \K}$,
 is infinite. By construction, the Frobenius of any unramified
prime in this quotient is torsion.
\end{proof}

\begin{rema}  For extensions for which  Frobenius elements  have {\it uniformly bounded} orders
see Checcoli \cite{Checcoli}. 
\end{rema}

\begin{rema}
Note that every $p$-adic analytic quotient of the infinite quotient that appears in Theorem \ref{theo:wilson} is finite: this is more or less obvious, due to the fact that an infinite $p$-adic analytic group has an open subgroup of finite cohomological dimension and is then torsion free (note here, we can do the same operation with $G_{S_p}$). 
\end{rema}

\subsection{The case of the center}

Using  Proposition \ref{prop:cutting}, one can also cut $\G_S$ by some special commutators. As we will see, this shows the limits of our method.

\medskip

Let $\G_S$ be as usual and let $\{a_1,\cdots, a_d\}$ be a minimal system of generator of $\G_S$ with Zassenhaus filtration $\omega_\G$. 
Let  $x$ be a non-trivial Frobenius element in $\G_S$. Then $\omega_\G([x,a_i]) \geq 1+ \omega_\G(x)$. Hence, assuming that $\G_S$ passes the 
Golod-Shafarevich test ($r<d^2/4$), we are guaranteed that when $\omega_\G(x)$ is large then $\Gamma:=\G_S/\langle [x,a_i], i=1,\cdots d  \rangle^\N$ is also infinite. 

\begin{prop} 
The class of the Frobenius element $x$ in $\Gamma$ is non-trivial and is in the center $Z(\Gamma)$ of $\Gamma$. 
\end{prop}

\begin{proof}
In $\Gamma$, the  class of  $x$ commutes with the class of $a_i$, for $i=1,\cdots, d$,   and thus with every element 
as the $a_i$'s topologically generate  $\Gamma$. That  $\omega_\G( [x,a_i])> \omega_\G(x)$ implies
$x$ is not trivial in $\Gamma$.
\end{proof}

Now let us remark that $\langle [x,a_i], i=1,\cdots,d \rangle \subset  \langle x \rangle^\N$, hence 
$$\Gamma:=\G_S/\langle [x,a_i], i=1,\cdots,d \rangle^\N \twoheadrightarrow \Gamma':=\G_S/\langle x \rangle^\N.$$

Here for the infiniteness of $\Gamma$ one has to check if
\begin{eqnarray} \label{eq:test1} 1-dt+rt^2+dt^{1+k}
\end{eqnarray}
has a root in $]0,1[$.
For the quotient $\Gamma'$, one has to check if
\begin{eqnarray} \label{eq:test2} 1-dt+rt^2+t^{k}
\end{eqnarray}
has a root in $]0,1[$. Some easy algebra shows that 
 $(\ref{eq:test2})$ is stronger than $(\ref{eq:test1})$: in other words, to prove that $\Gamma$ is infinite it is better to use the criteria for $\Gamma'$. 
Indeed, when $(d,r)=(9,21)$ and $k=3$,  the polynomial $1-9t+20t^2+t^3$ has a root in $]0,1[$ but $1-9t+20t^2+9t^4$ does not, so
 the Golod-Shafarevich test gives $\Gamma'$ is infinite and we can only conclude that $\Gamma$ is infinite as it has $\Gamma'$ as a quotient.
Note:

\begin{prop} Suppose that all primes in $S$ are coprime to $p$.
The pro-$p$ group $\Gamma$ is infinite if and only if $\Gamma'$ is infinite.
\end{prop}
 
\begin{proof}
Clearly $\# \Gamma' = \infty \implies \# \Gamma =\infty$.

Let $\N$ and $\N'$ be the kernels of the maps
$\G_S \twoheadrightarrow \Gamma$ and $\G_S \twoheadrightarrow \Gamma'$.
If $\Gamma'$ is finite   then $\K^{\N'}_S$ is a number field and 
$\K^{\N}_S/\K^{\N'}_S$ is a finitely generated  tamely ramified abelian  $p$-extension.
By class field theory such extensions are always finite so $\K^\N_S/\K$ is finite and thus $\Gamma$ is finite.
\end{proof}
 
 This situation shows that some cuts may be not optimal.

\subsection{Main result}
\begin{defi}
Let $\K$ be a number field and let  $\L/\K$ be a (possibly infinite) algebraic extension. 
The {\it root discriminant} of $\K$ is $\rd_\K := |\rm{Disc}(\K)|^{1/[\K:\mathbb Q]}$. 
The {\it root discriminant} of $\L/\K$ is $\limsup_{\J} |\rm{Disc}(\J)|^{1/[\J:\K]}$ where $\L\supset \J\supset \K$ and
$[\J:\K] < \infty$.
\end{defi}

\begin{defi}
An infinite extension $\L/\K$ is called {\it asymptotically good}  if its root discriminant is finite. 
\end{defi}

Using Proposition \ref{prop:cutting}, we will exhibit  
an asymptotically good extension ${\tilde \K}/\K$
  in which the set of primes that split completely is infinite. 
 \smallskip

\begin{theo} \label{theo:ihara} Let $\K$ be a number field, and $S$ be a finite set of places of $\K$ coprime to $p$.  Suppose  that $d_p \G_S > \alpha_{\K,S}$. Then there exists an  infinite  pro-$p$ extension ${\tilde \K}/\K$ in $\K_S/\K$ for which the  set of primes that split completely  is infinite. 
\end{theo}

\begin{rema}
As in Theorem~\ref{theo:wilson},  the pro-$p$ group $\Gal({\tilde \K}/\K)$ is not finitely presented.
\end{rema}

\begin{proof}
The proof is almost identical to that of Theorem~\ref{theo:wilson}.
Let $1\longrightarrow \R \longrightarrow \F \stackrel{\varphi}{\longrightarrow} \G_S \longrightarrow 1$ be a minimal presentation of $\G_S$.
By hypothesis 
$r<d^2/4$. 
  Take $P_{\PP}(t)=1-dt+rt^2$ as the Golod-Shafarevich polynomial, and note that $P_{\PP}(d/2r)=1-d^2/4r <0$.
   We will apply Proposition \ref{prop:cutting} $(ii)$ with $t_0=d/2r \in ]0,1[$ and $k'$ as given there. We will take the quotient by infinitely many Frobenii $x_i$
   of unramified primes whose depth is at least $k'+i$ in $\G_S$.
For $i\geq 2$, denote by $\G_i$ the image $\varphi(\F_i)$;  Proposition~\ref{lemma:Zassenhausfiltration} gives that  $\G_i$ is also the Zassenhaus filtration of $\G_S$.
Now, for $i \geq 0$,  choose a prime ideal $\p_i$ of $\O_\K$ such that  its Frobenius $x_i \in \G_S$ is  in  $\G_{k'+i}$  (in fact a conjugacy class there), and such that $\p_i\notin \{\p_0, \cdots, \p_{i-1}\}$. Choose $y_i\in \F_{k'+i}$ such that $\varphi(y_i)=x_i$
so 
$\omega_\F(y_i) \geq k'+i$.  
The quotient $\Gamma$ of $\G_S$ by the normal subgroup generated by the Frobenius~$x_i$ of the primes $\p_i$, $i\geq 0$ is 
$$\Gamma \simeq \G_S/\langle x_i, i \rangle^{\rm N} \simeq \F/\langle \R,y_i, i \rangle^{\rm N}.$$
Denote by ${\tilde \K} \subset \K_S$ the fixed field by $\langle x_i,\ i\geq 0 \rangle^{\rm N}$; $\Gal({\tilde \K}/\K) \simeq \Gamma$. By Proposition \ref{prop:cutting} $(ii)$, the pro-$p$ extension ${\tilde \K}/\K$ is infinite, and
each prime $\p_i$ has trivial Frobenius in~${\tilde \K}$ and thus splits completely.
\end{proof}

\begin{rema}
Theorem \ref{theo:ihara} and Theorem \ref{theo:wilson} are particulary interesting in the context of Tsfasman-Vladut  \cite{TV}. See also Lebacque \cite{Lebacque}.
\end{rema}

\subsection{The quantity of splitting}

\begin{defi}
For a (possibly infinite) Galois extension  $\L/\K$ of a number field $\K$, denote by 
$\SS_{\L/\K}=\{  \p \subset \O_\K \mid  \p \ {\rm a \ prime \ ideal,} \ \p  {\rm \ splits \ completely \ in  \ } \L/\K\}$, and for $X \geq 0$
 $$\displaystyle{\SS_{\L/\K}(X):= \{\p \in \SS_{\L/\K}, \ \N(\p)\leq X\}}.$$
Put
$$\pi_{\L/\K}(X)=|\SS_{\L/\K}(X)| \cdot$$
\end{defi}

The effective version of Chebotarev's Theorem allows us to  give an upper bound for  $\pi_{\L/\K}(X)$ when the extension
 $\L/\K$ is asymptotically good. Indeed:

\begin{prop}
If  $\L/\K$ is asymptotically good, then (assuming the  GRH), 
$$\pi_{\L/K}(X) \ll X^{1/2}\log X \cdot$$ 
\end{prop}

\begin{proof}
Pass to the limit  Theorem 4 of \cite[\S 2.4]{Serre}.
\end{proof}

Recall an another  estimate given by   Ihara \cite{Ihara}. See also Tsfasman-Vladut \cite{TV}.

\begin{theo} \label{theo-Ihara}
Let $\L/\K$ be an infinite asymptotically good extension. Then (assuming the GRH), 
 $$\lim_{X \rightarrow \infty} \sum_{\p \in \SS_{\L/\K}(X)}\frac{\log \N(\p)}{\sqrt{\N(\p)}}   < \infty \cdot$$
\end{theo}

Suppose $\G_S$ tame and  infinite. Denote by $\G_i$ the Zassenhaus filtration of $\G_S$. Suppose moreover that  $r<d^2/4$: the pro-$p$ group $\G_S$ is  not analytic and  then $\G_i \neq \G_{i+1}$ for all $i$ (see \cite[Chapter 11, Theorem 11.4]{DSMN}). 

\begin{rema} When $\G_S$ is tame and infinite,  by the tame Fontaine-Mazur conjecture \cite[Conjecture (5a)]{FM} $\G_S$ must not be  analytic and then $\G_i \neq \G_{i+1}$ for all $i$.
\end{rema}

\begin{defi} Let  $\G:=\G_S$ tame and infinite. For a prime $\p \notin S$, denote by  $x_\p$ the Frobenius at~$\p$ in~$\G_S$.
Define for $i\geq 1$, $$\N_i:=\min \{\N(\p), x_\p \in \G_i\backslash \G_{i+1} \}.$$
\end{defi}

Recall  pro-$p$ extensions of a number field that are tamely ramified at a finite set of places are always asymptotically good.
One can produce some asymptotic good extensions where the set of splitting is infinite, 
and in particular, by our construction, the  series   $\displaystyle{\sum_{i \geq 2} \frac{\log \N_i}{\sqrt{\N_i}}}$ converges.

\begin{theo}
Let $\K_S/\K$ be a pro-$p$ and tame extension for which $d_p\G_S \geq \alpha_{\K,S}$. Then, considering the Zassenhaus filtration $\G_i$ of $\G:=\G_S$, one has along the tower $\K_S/\K$  the estimate: $$\N_i \gg i^{2}.$$
\end{theo}

One can say more when $d < r$ and $r<d^2/4 $, that is when $2\sqrt{r} <d< r$.

\begin{defi} Set $\G:=\G_S$ tame and infinite, and for a prime $\p$ denote by  $x_\p$ the Frobenius at $\p$ in $\G$.
Define for $i\geq 1$, and $k \in \Z_{\geq 1}, $ $$\N_i^{(k)}:={\rm the \ kth \ smallest \  norm \ of \ a \ prime \ } \p {\rm \ with} \  x_\p \in \G_i \backslash \G_{i+1}.$$ 
Of course, $\N_i^{(1)}= \N_{i}$.
\end{defi}

\begin{theo} Assume the GRH. 
Let $\G_S$ be the Galois group of a tame $p$-tower for which~$r<d^2/4 $. Choose $\lambda$ and $m$ as for  Lemma \ref{prop:cutting2}.  Put $\beta_{k,m}:= \lambda^{k+m}$.
Then $$\N_k^{(\beta_{k,m})} \gg \lambda^{2k}.$$ 
\end{theo}

\begin{proof}  Set $\G:=\G_S$. Observe that here $r -d \geq 0$.
For $k\geq 0$, let us choose $\lambda^{k+m} $ different prime ideals $\p_{i,k} \subset \O_\K$ (of smallest norm as possible) such that $x_{\p_{i,k}} \in \G_{m+k}\backslash \G_{m+k+1}$. 
The element $x_{\p_{i,k}} $ is  of depth $m+k$.
Denote by $\tilde{\K}:=\K_S^{\langle \varphi(y_{i,k}),\ i,k\rangle^{\rm N}}$.
Then  Lemma~\ref{prop:cutting2} implies that  $\tilde{\K}/\K$ is infinite: it is   an asymptotically good extension where each prime $\p_{i,k}$ splits completely.
Put $\beta_{k,m}:= \lambda^{k+m}$.
Then by the estimation of Ihara (Theorem~\ref{theo-Ihara}) for 
$\tilde{\K}/\K$, one has: $$\sum_{k \geq 0} \lambda^{k+m} \frac{\log \N_k^{(\beta_{k,m})}}{\sqrt{\N_k^{(\beta_{k,m})}} }< \infty,$$
which implies that $\N_k^{(\beta_{k,m})} \gg \lambda^{2k+2m}$.
\end{proof}


\section{The constants of Martinet}
In this section we set new records for root discriminants in asymptotically good totally complex and totally real towers.

\subsection{Tame towers with finite ramification-exponent} 
We will  again use Proposition \ref{prop:cutting}. The set $S$ will consist of $\p$
 with $\N(\p) \equiv 1 \, {\rm mod} \ p$.  Recall that $d=d_p H^1(\G_S)=d_p \G_S$ is the $p$-rank of $\G_S$ and  $r=d_p H^2(\G_S)$ is the minimal  number of relations of $\G_S$.

\begin{defi} Fix $k\geq 1$.
Denote by $\K_S^{[k]}/\K$ the maximal pro-$p$ extension of $\K$ unramified outside $S$ and where the exponent  of ramification at $\p\in S$ is at most $p^{k}$
so $\K_S^{[\infty]} =\K_S$.
Put $\G_S^{[k]}:=\Gal(\K_S^{[k]}/\K)$.
\end{defi}

\begin{rema}
The extension $\K_S^{[k]}/\K$ is well-defined because  inertia  groups  are  cyclic in the tame case.
\end{rema}

\begin{prop} \label{prop;exponent}
Assume that $r<d^2/4 $. Put $k_0=\lceil  \log ( \frac{d^2}{4r}-1)/\log(d/2r)\rceil$.
 Then, for $k\geq \log_p(k_0)$, the extension $\K_S^{[k]}/\K$ is infinite.
\end{prop}

\begin{proof} 
We follow the notations of Proposition \ref{prop:cutting}.
We have chosen $k_0$ so that $P_{\PP'}(t) = 1-dt+rt^2+t^{k_0}$ is negative at $t=d/2r$. 
Take as  $x$ a generator of the inertia group  at $\p$ in $\K_S/\K$, cut by $x^{p^k}$ and
apply Proposition \ref{prop:cutting} $(i)$.
\end{proof}

Recall the root discriminant of a number field $\K$ is denoted by  $\rd_\K$.
  The interest of extensions as above is the following:
  
  \begin{prop}
  In the tower $\K_S^{[k]}/\K$ the root discriminant is bounded by $$\rd_\K \cdot  \big(\N(\p)^{\frac{1}{[\K:\Q]}}\big)^{1-\frac{1}{p^{k}}}.$$  
  \end{prop}
  
\begin{proof}    The result follows from the basic theory of tame ramification.
\end{proof}  
  In
  \cite{Hajir-Maire-Compositio} it is shown, by taking the limit in the above Proposition, that the root discriminant of
  $\K_S/\K= \K_S^{[\infty]}/\K$
   is bounded by $\rd_\K \cdot  \big(\N(\p)^{\frac{1}{[\K:\Q]}}\big).$

 We can now give  an answer to a central question of \cite{Hajir-Maire-JA}:
 
 \smallskip
 
\begin{theo}\label{theorem:unramifiedsubtower}
Suppose  $S\neq \emptyset$ such that $d_p \G_S  > \alpha_{\K,S}$. Then there exists a finite extension  $\L/\K$ in $\K_S/\K$ such that $\L_\emptyset/\L$ is infinite. 
\end{theo}

  \begin{proof} Observe that {\it wlog} we can assume that $S=\{\p\}$ contains only one prime.
  By hypothesis, $r<d^2/4$, so for large $k$, the extension $\K_S^{[k]}/\K$ is infinite. The  inertia  group  at~$\p$ is a quotient of $\Z/p^k\Z$. By changing the base field, there exists a finite extension $\L/\K$ such that $\K_S^{[k]}/\L$ is unramified and infinite.
\end{proof}

   \subsection{Some set up}

Let $\K$ be a number field and $S$ a finite set of  finite  places of $\K$. For $v\in S$ let $\K_v$ be the completion
of $\K$ at $v$ and recall $\delta_{\F,p}=1$ if $\F$ contains  the $p$th roots of unity and
$\delta_{\F,p}=0$ otherwise.

\smallskip
Let $$\V_{S} =\{x \in \K^\times \mid x \in \K^{\times p}_v \mbox{ for } v\in S, \  \mbox{and } v(x) \equiv 0\mbox{ mod } p ,\forall \, v \}$$
and let  $\CyB_S$ to be the character group of $\V_{S}/\K^{\times p}$.
Recall the exact sequence
$$ 0\to \sha^2_S \to H^2(\G_S) \to \oplus_{v\in S}H^2(G_{\K_v})$$
where each term on the right is just  $\Z/p\Z$ or $0$ depending on $\delta_{\K_v,p}=1$ or $0$; observe also that  when $\delta_{\K,p} \neq 0$,  global reciprocity implies
the image of the right map lies in the hyperplane of terms that sum to zero.

\smallskip
From Chapter 11 of \cite{Koch} we know 
\begin{equation}\label{eq:h1rank}
d_p \G_S =\left(  \sum_{v\in S_p}[\K_v:\Q_v]  \right) -\delta_{\K,p} +\left( \sum_{v\in S} \delta_{\K_v,p} \right) -(r_1+r_2) +1 +d_p (\CyB_S)
\end{equation}
and there is a natural injection
$  \sha^2_S \hookrightarrow \CyB_S$ which is an isomorphism if
$S$ contains all primes of $\K$ dividing $p$.

\begin{rema}
Numerically showing the injection above is {\it not} an isomorphism
in explicit tame cases would likely lead to strong improvements in root discriminant bounds in asymptotically good towers.
\end{rema}

When $S$ is tame  software will allow us to explicitly compute $d_p \G_S$ in many cases, thus giving $d_p  \CyB_S$  
exactly and the upper bound
$ r(\G_S) \leq \left\{ \begin{array}{cc} d_p \CyB_\emptyset & S=\emptyset \\ d_p \CyB_S +|S| -\delta_{\K,p} & S \neq \emptyset \end{array}\right. .$

 \subsection{Examples and records} 
 For various computations of $H^1$s and ray class groups we have used the software packages  PARI/GP  \cite{pari}  and MAGMA \cite{Magma}.
 We take always $p=2$ in this subsection.

\subsubsection{An example of J. Martin}
  In his Ph.D. thesis, \cite{Martin}, Martin found a degree $8$ totally real number field $\K$ whose $2$-class group has rank $8$. Equation~\eqref{eq:h1rank}
  gives that $\dim \CyB_\emptyset =16$ so $ \dim H^2(\G_\emptyset) =\dim \sha^2_\emptyset  \leq \dim \CyB_\emptyset =16$. The Golod-Shafarevich polynomial is 
  (at worst) $P(t)=1-8t+16t^2$. Note $P(1/4) = 0$  so $\G_\emptyset$ is infinite. As Martin's thesis is unpublished,
  we record his polynomial here: $x^8 - 3297x^6 + 14790x^5 + 3555341x^4 - 24457440x^3 - 1347361755x^2 + 7744222350x +149856133975$.
  Its discriminant is $3^4\cdot5^4\cdot7^4\cdot 13^2\cdot 29^4 \cdot 53^2 \cdot 109^2$ and the root discriminant is less than
  $913.4927$.
  \color{black}
  \subsubsection{The totally complex case}
  
  \

  $\bullet$ Take $\K=\Q(\sqrt{13},\sqrt{-3 \cdot 5 \cdot 17})$. Software gives
  $d_2 \G_\emptyset=4$ so Equation~(\ref{eq:h1rank}) gives $d_2  \CyB_\emptyset =6$. The Golod-Shafarevich polynomial
  $1-4t+6t^2$ has no root so we cannot conclude $G_\emptyset$ is infinite. 
  There are two primes above $43$ in $\K$, both having norm $43^2$.
 Take $S$ to be either of these. We write $S=\{\p_{43^2}\}$.
 Software gives that $d_2 \G_S=5$ so $d_2 \CyB_S =6$ and $r(\G_S)  \leq 6+1-1=6$. In this case the 
   Golod-Shafarevich polynomial
  $1-5t+6t^2$ has a root in $]0,1[$.
  As $d_2 \G_\emptyset=4 < 5 = d_2 \G_S$,
the generator of inertia $\tau_{\p_{43^2} }\in \G_S$ has depth~$1$ by  Proposition \ref{corollary:detect} $(iii)$.  We cut $\G_S$ by the relation $\tau_{\p_{43^2}}^4$ which has
depth at least~$4$ by Proposition~\ref{prop:proprietes-omega}.  As $1-5t^2+6t^2+t^4$ has a root in $]0,1[$, the group   $\G_S^{[2]}$ is infinite, and in this tower one has: 
$$\rd_{\K_S^{[2]}} = \rd_\K \cdot (43^2)^{\frac14 \cdot (1-\frac14)}   < 235.9351.$$
  This is not close to the record of $82.1004$ in  \cite{Hajir-Maire-JSC} .

\

  $\bullet$
  Take the number field $\K$
  with polynomial $x^{12} + 138x^{10} - x^9 + 6579x^8 - 1191x^7 + 142088x^6 - 78327x^5 + 1495530x^4 - 1492094x^3 + 8549064x^2 - 6548187x + 27239851$.
  The field is totally complex and software gives
  $$\rd_\K < 75.7332,\,\,d_2 G_\emptyset=7\mbox{ and }d_2  \CyB_\emptyset =13.$$ 
  The Golod-Shafarevich polynomial
  $1-7t+13t^2$ has no root so we set $S$ to be the one
  prime $\p_{9}$ above $3$. It has norm $9$ and software gives $d_2 \G_{S}=d_2 \G_\emptyset=7$ so $\tau_{\p_{9}}$ 
  is of  depth at least~$2$ by  Proposition~\ref{corollary:detect} $(iii)$  and
 $\tau^2_{\p_{9}}$ has depth at least $4$.
  One sees that $d_2 \CyB_S=12$ and $r(\G_S) \leq 12$. As $1-7t+12t^2$ has a root in  $]0,1[$, $G_S$ is infinite. 
After cutting by $\tau_{\p_{9}}$, our Golod-Shafarevich polynomial
  $1-7t+12t^2+t^4$ has a root in  $]0,1[$ so $\G_S^{[1]}$  is infinite and
  $$\rd_{\K_S^{[1]}} = \rd_\K \cdot (9)^{\frac1{12} \cdot (1-\frac12)}   < 82.9940.$$
 This is quite close to the record of $82.1004$ of \cite{Hajir-Maire-JSC}.

 \
  
  $\bullet$  
  In this example we establish a new record by cutting the old one.
  Consider the totally complex number field $\K$ of degree $12$ in \cite{Hajir-Maire-JSC} with 
polynomial $x^{12}+339x^{10}-19752x^8-2188735x^6+284236829x^4+4401349506x^2+15622982921$.
 Let $\HH$  be the Hilbert Class Field of 
$\K$.
Software yields that $\Gal(\HH / \K) \simeq (\Z/2)^6$ so 
$$d_2 \G_\emptyset=6,\,\, d_2 \CyB_\emptyset = 6+r_2=6+6=12 \mbox{ and } r(G_\emptyset) \leq 12.$$
The polynomial $1-6t+12t^2$ is always positive so we cannot conclude
that the maximal pro-$2$ quotient of $\G_\emptyset$ is infinite. Here $\rd_{\K} < 68.3636$.
 Now take $S=\{\p_{9}\}$, the unique prime above $3$ of norm $9$. Software gives  $d_2 \G_S=7$ so $d_2 \CyB_S=12$ and 
 we have the bound 
 $r(\G_S) \leq 12+1-1=12$.  The polynomial $1-7t+12t^2$ has a root in $]0,1[$ so $G_S$ is infinite.
   As $d_2 \G_\emptyset=6 < 7 = d_2 \G_S$,
 $\tau_{\p_{9}}$ has depth $1$. We cut by $\tau^4_{\p_{9}}$ to get 
 Golod-Shafarevich  polynomial
  $1-7t+12t^2+t^4$ which has a root in $]0,1[$
so 
$\G_S^{[2]}$ is infinite, and in this tower 
$$\rd_{\K_S^{[2]}} = \rd_\K \cdot (9)^{\frac1{12} \cdot (1-\frac14)}   < 78.4269.$$
  This is a new record with savings  a factor of  $3^{1/24} \approx 1.04683 \dots$
  
  \

  \subsubsection{The totally real case}
 
 \
 
 $\bullet$  We establish a new record here as well.
 Let $\K$ be the  totally real field of \cite{Hajir-Maire-JSC} of degree $12$ over $\Q$.
 It's polynomial is $x^{12}-56966x^{10}+959048181x^8-5946482981439x^6+14419821937918124x^4-12705425979835529941x^2+3527053069602078368989$
 and $rd_K < 770.6432$.  
 All primes above $13$ in $\K$ have norm $13$. Take $S$ to be any one of them. Software gives 
$$d_2 \G_\emptyset=d_2 \G_S=9,\,\, d_2 \CyB_\emptyset =21 \mbox{ and } d_2 \CyB_S =20.$$
The Golod-Shafarevich polynomial for $\G_\emptyset$ is 
  $1-9t+21t^2$ and has no root in  $]0,1[$ so we cannot conclude $G_\emptyset$ is infinite, though we suspect it is.
The Golod-Shafarevich polynomial for $\G_S$ is  
 $1-9t+20t^2$ which  has a root in  $]0,1[$.
 As $d_2 G_\emptyset = d_2 \G_S=9$ we see $\tau_{\p_{13}}$ has depth at least $2$ by Proposition \ref{corollary:detect} $(iii)$.  We cut by $\tau^2_{\p_{13}}$ which has depth  at least $4$. 
  As $1-9t+20t^2+t^4$ has a root in  $]0,1[$ 
  $\K_S^{[1]}/\K$ is infinite and
  $$\rd_{\K_S^{[1]}} = \rd_\K \cdot (13)^{\frac1{12} \cdot (1-\frac12)}   < 857.5662.$$
This is a new record with savings by a factor of $13^{1/24} \approx 1.11279 \dots$

\subsubsection{Comments}

In the example above, a hope would be that  $\tau_{\p_{13}}$ has depth greater than two in $\G_S$. In that case we could
 cut by  the relation $\tau_{\p_{13}}$ and the corresponding Golod-Shafarevich polynomial would be at most
  $1-9t+20t^2 +t^3$ which has a root in $]0,1[$. One would then have that $\K$ has infinite $2$-Hilbert Class Field
  Tower and the totally real root discriminant record would  be $< 770.644$. We do not see how to check the depth
  of $\tau_{\p_{13}}$  in $\G_S$. See also the beginning of $\S 5$.  
\smallskip

The totally complex record was $82.1004$ with a GRH lower bound of $8\pi e^{\gamma} \approx 44.763$.
For the totally real case, the record was  $913.4927$  and the GRH lower bound is $ 8\pi e^{\gamma +\frac{\pi}2} \approx 215.33$.
One should probably take the ratio and then logs to  measure distance to the GRH bounds. 
Then for a number field $\K$, let us define $\partial(\K)=\log(\Rd_\K /\alpha)$ where $\alpha = 44.763$ if $\K$ is totally imaginary or $\alpha=215.33$ if $\K$ is totally real.
Let us recall the different improvements. The ordered pairs in the table below are $(rd_\K,\partial)$.

\medskip

\begin{mdframed}[style=MyFrame]
\begin{center}
\begin{tabular}{l|c|c|c|c}
Signature & Martinet  (1978) & Hajir-Maire (2002) &  Martin (2006) & new records \\
\hline
tot. compl.   &$(92.368; \  0.7244)$ & $(82.1004; \ 0.6066)$ &  & $(78.427; \ 0.5608)$  \\
\hline
tot. real  &$(1058.565;   \  1.592)$ &$(954.293; \  1.488)$  &   $(913.493; \  1.445)$ &  $(857.567;  \  1.382)$  \\
\hline
\end{tabular}
\end{center} 
\end{mdframed}

\medskip 

The recent improvement of $\partial$ in the totally imaginary case  is  $ 7.55\% $, and $4.36\% $ for the totally real case. 
\color{black}


\section{Cutting of wild towers}
  
  \subsection{Local abelian extensions}
   \label{section:principle}
  
 For this section, our results follow from this main observation: {\it we can cut  wildly ramified towers if  we first cut by local commutators}.
  We also assume throughout this section that in our wild extensions, the assertion of Kuz'min's Theorem holds, that is
 the  pro-$2$ local Galois groups above $(2)$ are maximal. In the first totally complex example  of \S \ref{section:paragraph4.2} the hypotheses of Kuz'min's Theorem 
 are satisfied, but we do not include the infinite places in the totally real example. It is possible that less cutting is needed
 in the latter example.

  \begin{defi} Take $S=S_p$ the set of $p$-adic places.
  Denote by $\K_S^{{[k],p-ab}}/\K$ the maximal pro-$p$ extension unramified outside $S$, locally abelian at $p$ (and then at all places),  and for which  the inertia groups at $v|p$ are  of exponent dividing $p^k$. Put $\G_S^{{[k],p-ab}}=\Gal(\K_S^{{[k],p-ab}}/\K)$.
  \end{defi}
  
  Recall that for $S=S_p$, the pro-$p$ group $\G_S$ is of cohomological dimension $2$ and $r(\G_S)=d(\G_S)-r_2-1$. (For $p=2$, $S$ must
  contain all the infinite places, a vacuous condition in the totally complex case).

  \

\begin{theo}\label{prop:boundat2} Take $p=2$ and $S=S_p$.
In $\K_S^{[1],p-ab}/\K$, the root discriminant is bounded by 
$ \rd_\K  \cdot  2^{   \frac{\sum_{v|p}{f_v} \left( 2+\frac1{e_v}-\frac{1}{2^{e_vf_v}} \right)  }{[\K:\Q]}     }.$
\end{theo}
\begin{proof} 
Fix a place $v|2$ of $\K$. 
By Kummer's theory, the quadratic extensions of $\K_v$ are parametrized by the classes of 
$\K^\times_v/{\K^\times_v}^2 \simeq \langle  \pi_v \rangle/ \langle \pi_v^2 \rangle \times U_v/U^2_v$ where the latter factor has dimension
$e_vf_v+1$ over $\mathbb F_2$ so the maximal elementary abelian $2$-extension of $\K_v$ has degree $2^{e_vf_v+2}$. We compute its discriminant
over $\K_v$ by using the conductor-discriminant formula, namely we take the product of the conductors of {\it all} quadratic characters
of~$\K_v$.  Note there is exactly one character for each quadratic extension, so the discriminant {\it equals} the conductor in this case.
It is elementary to compute an upper bound on the discriminant of a quadratic field, so by taking the product over all quadratic fields we obtain an
upper bound for our local  discriminant.
There are $2^{e_vf_v+2}-1$ quadratic extensions of $\K$: 
\begin{itemize}
\item[$\bullet$]  $2^{e_vf_v+1}- 1$ extensions corresponding to extracting the square root of a unit.  These have conductor  dividing  $4=\pi^{2e_v}$. One extension  is unramified 
        and has conductor $1$. 
\item[$\bullet$] $2^{e_vf_v+1}$ extensions corresponding to extracting the square root of the uniformizer times a unit. These have conductor $\pi^{2e_v +1}$.
\end{itemize}
For more details see for example \cite[Chapter II, Proposition 1.6.3]{gras}.
Thus the discriminant of the maximal elementary abelian $2$-extension of $\K$ divides
$$(\pi^{2e_v})^{2^{e_vf_v+1}-2} \cdot (\pi^{2e_v+1})^{2^{e_vf_v+1}} = \pi^{e_v\cdot2^{e_vf_v+3} +2^{e_vf_v+1}-4e_v}.$$
Taking the $2^{e_vf_v+2}$th root, we get the root discriminant is
$$(\pi)^{2e_v+\frac12-\frac{e_v}{2^{e_vf_v}}} =(2)^{2+\frac1{2e_v} -\frac1{2^{e_vf_v}}}.$$
This is the local contribution. The norm of $v$ is $2^{f_v}$ so in the global root discriminant we get a the factor
$2^{f_v \left( 2+\frac1{e_v}-\frac{1}{2^{e_vf_v}} \right)   }$. Now sum over $v|p$ and take the $1/[\K:\Q]$th root.
\end{proof}

\begin{rema}
 One has $\G_{S_p}/\langle [D_v,D_v], v|p \rangle \twoheadrightarrow \G_{S_p}^{ab}:=\G_S/[\G_S,\G_S]$ and then the maximal local abelian 
extension  $\K_{S_p}^{p-ab}/\K$ of $\K_{S_p}/\K$  is infinite (it contains the cyclotomic extension). In order to have a criteria proving that $\K_S^{{[k],p-ab}}/\K$ is infinite, we 
 need a Golod-Shafarevich polynomial of $\K_S^{p-ab}/\K$  to have a root in $]0,1[$. 
  \end{rema}

\subsection{Examples}  \label{section:paragraph4.2}

\subsubsection{} Take $\K=\Q(\sqrt{-5460})$ and $p=2$. Software gives $\G^{ab}_\emptyset \simeq (\Z/2)^4$, and for $S=\{\p_2\}$, the unique prime above $2$,  
$\G_S^{ab} \simeq \Z_2^2 \times (\Z/2\Z)^3$.  Also $\p_2$ is not principal and thus its Frobenius in $G_\emptyset$ has depth $1$.

In this  totally complex wild case, global duality implies the natural injection $\sha^2_S \hookrightarrow \CyB_S$ from \S 3.2 is an isomorphism and
 $r= r(G_S) = d-1-r_2=3$.

 Recall  that the decomposition group at $\p_2$ in $\G_S$ has at most $4$ generators, as a quotient of $\G_\p=\Gal(\overline{\K}_\p/\K_\p)$  this last 
 group being isomorphic to the Demushkin group with $4$ generators (here $\overline{\K}_\p$ is the maximal pro-$p$ extension of the complete field $\K_\p$). Denote by $\langle x,y,z,t \rangle$ these generators viewed in $\G_S$.
 The structures of  $\G^{ab}_\emptyset=\G_\emptyset/[\G_\emptyset,\G_\emptyset]$ and of $\G_S^{ab}$  show that the elements $x,y,z,t$ can be choosen such that 
 $t$ (Frobenius) has depth $1$ as does $x$, a generator of inertia.  The other variables, $y$ and $z$,  have depth at least $2$.
 Then
 \begin{enumerate}
 \item[$\bullet$]  $[x,t]$ has depth at least $2$,
 \item[$\bullet$] $[x,y]$, $[x,z]$, $[t,y]$, $[t,z]$ have depth at least $3$,
 \item[$\bullet$]  and $[y,z]$ has depth at least $4$. 
 \end{enumerate}
 
 If we cut $\G_S$ by the local commutators, the Golod-Shafarevich polynomial to test becomes:
 $$P(t)=(1-5t+3t^2 ) +(t^2+4t^3+t^4) =   1-5t+4t^2+4t^3+t^4,$$
 which has a root in  $]0,1[$.
Then,  $\G_{S}^{p-ab}:=\Gal(\K_{S}^{p-ab}/\K)$, the maximal pro-$2$ extension of~$\K$ unramified outside $2$ and locally abelian is infinite 
(in fact not $p$-adic analytic).
Now, we can apply Proposition \ref{prop:cutting} to cut the ramification at a certain depth. Here   $\G_S^{[2],p-ab}$ is
 infinite: indeed  the polynomial  $1-5t^2+4t^2+4t^3+4t^4$ has a root in $]0,1[$.
 
 In this case the root discriminant in $ \K_{S}^{2,p-ab}/\K$ is bounded by 
 $\rd_\K \cdot 2^{9/8} <  161.1592  $,
 thanks to Theorem~\ref{prop:boundat2}.

\subsubsection{} In  this example we demonstrate the  effectiveness of using a mixed strategy of simultaneous tame and wild cutting.  Take $\K$ to be  the Hilbert Class  Field of the cyclic cubic extension of conductor $163$. The number field $\K$ is of degree $12$
 over~$\Q$ with equation given by $x^{12} - 23 x^{10} + 125x^8 - 231x^6 + 125x^4 - 23x^2 + 1=0$.  Here for $v|2$, $f_v=3$.
 Take $S=\{v|2, 3, 5, 7\}$.  
 Software and some basic theory give
 \begin{enumerate}
  \item[$\bullet$] there are $4$ primes in $\K$ above above $2$, $3$ and $7$. There are $6$ primes above $5$.
  \item[$\bullet$]   $d_2 \G_S=18$, and  Equation~(\ref{eq:h1rank}) implies  $\CyB_S=0$ so $\sha^2_S=0$ as well
  \item[$\bullet$]  $r(G_S) \leq 17$ so the Golod-Shafarevich polynomial is $1-18t+17t^2$, which is {\it very} negative on $]0,1[$.
  \end{enumerate}
We are going to cut by
 \begin{enumerate}
  \item[$\bullet$] the local commutators of each place above $2$, i.e. by $4\cdot 10$ elements of depth at least~$2$,
  \item[$\bullet$] the square of the generators of the abelian local inertia at the wild places (observe that we can take a generator of order $2$), i.e by $12$ elements of depth 
  at least $2$,
  \item[$\bullet$] the fourth of the generators of the  inertia at three  places dividing $5$,
  \item[$\bullet$] the square of the generators of the inertia of the other eleven places dividing $3\cdot 5\cdot 7$.
 \end{enumerate}
 
 Then the pro-$p$ group of the new quotient has $1-18t+(17+40+12+11)t^2+3t^4$ as polynomial that  has a root in $]0,1[$. Here, one has $$\rd \leq 163^{2/3} \cdot (3 \cdot 7 \cdot 5^{1/2})^{1/2} \cdot (5^{1/2})^{3/4}\cdot 2^{23/8} \color{black}< 2742.95621 \cdots$$
 
 \subsection{Cutting a $p$-rational tower}
 
 \subsubsection{$p$-rational field}
 
 Let us recall the notion of $p$-rational field (see for example \cite{Gras-Jaulent}, \cite{Mova-Nguyen}).
 
 \begin{defi}
 A number field $\K$ is called $p$-rational if the maximal pro-$p$ extension of $\K$ unramified outside $S_p$ is free pro-$p$.
 \end{defi}

In the context of the inverse Galois problem, this notion is also very useful for producing some special pro-$p$ extensions of number fields: see Greenberg \cite{Greenberg},
Hajir-Maire  \cite{Hajir-Maire-mu}, etc.

An easy argument from group theory gives:
 
\begin{prop} \label{prop:prationalheridity}
Let $\K$ be a $p$-rational field and let $\L/\K$ be a finite extension in $\K_{S_p}/\K$. Then $\L$ is $p$-rational.
\end{prop}

 Assuming Leopoldt's conjecture, it is well-known that $\G_{S_p}$ is pro-$p$ free if and only if $\G_{S_p}^{ab}$ is torsion free.
   The torsion of $\G_{S_p}^{ab}$ can be estimated by class field theory: in particular for $p$ sufficiently large this torsion is isomorphic to the $p$-part of $\left(\prod_{v\in S_p} U_v \right)/\overline{\O_\K^\times}$ which is easy to compute. 
   After many  observations Gras \cite[Conjecture 7.11]{Gras-CJM} made recently  the following conjecture:
   
 \begin{conjecture}[Gras]
 Given  a number field $\K$, then for large $p$,  $\K$ is $p$-rational.
 \end{conjecture}

\smallskip

We use Gras' Conjecture to  produce  $p$-rational number fields $\L$ with large $p$-class group. 
  
\subsubsection{Results}  
  
  First, we obtain:

\begin{theo} \label{theo:unramified-prational}
Let $\K$  be a totally imaginary field of degree at least $12$ over $\Q$. Choose $p>2$ such that : $(i)$ $p$ splits completely in $\K/\Q$ and, $(ii)$ $\K$ is $p$-rational.
Then there exists a number field $\L/\K$ in $\K_{S_p}/\K$ such that $\L_\emptyset/\L$ is infinite.
Note that as Gal$(\L_{S_p}/\L)$ is a subgroup of the free pro-$p$ group Gal$(\K_{S_p}/\K)$ then $\L$ is $p$-rational.
\end{theo}

 \begin{proof} 
 Let $p>2$. As $p$ splits completely in $\K$,
$\mu_p \not \subset \K_v$  and $\G_v$ is a free pro-$p$ group on $2$ generators $x_v,\tau_v$, where $x_v$ can be chosen as  the Frobenius and  $\tau_v$ as a generator of the  inertia group.

Suppose moreover that $\G:=\G_{S_p}$ is $p$-rational.
Then we cut the free pro-$p$ group $\G$ on $r_2+1$ generators by all the commutators $[x_v,\tau_v]$ for all $v|p$, to obtain a pro-$p$ extension $\K_{S_p}^{p-ab}/\K$ corresponding to the maximal local abelian extension at every $v|p$ of $\K$ unramified outside $p$;  here $r_2$ is the number of complex embeddings of $\K$. 
Put $\Gamma:=\Gal(\K_{S_p}^{p-ab}/\K)$. A naive presentation $(\PP)$ of $\Gamma$ allows us to obtain the Golod-Shafarevich
 polynomial   $$P_\PP(t)=1-(r_2+1)t+2r_2t^2.$$
As $r_2 \geq 6$, we easily compute  that $P_\PP$ is negative on $]0,1[$,  and so  for a large given $k$, if we  cut by the powers 
$\tau_v^{p^k}$ of $\tau_v$, $v|p$, the extension  $\K_S^{{[k],p-ab}}/\K$ is infinite. As $\Z/p^k$ maps onto the inertia group of all $\p|p$, one concludes by a changing the base field.
\end{proof}

Recall that if a pro-$p$ group $\G$ passes the test of Golod-Shafarevich, then $\G$ is not $p$-adic analytic (when $d\geq 3$) and,  by Lubotzky-Mann \cite{Lubotzky-Mann}, the $p$-rank of the open subgroups $U$ of $\G$ tends to infinity with $[ \G:U]$. 
 In fact, one has the following due to Jaikin-Zapirain (see \cite{Jaikin} or \cite[Theorem 8.3]{Ershov}):
 
 \begin{theo}[Jaikin-Zapirain] \label{theo:jaikin}
 Suppose that a pro-$p$ group $G$ passes the Golod-Shafarevich  test. Then
 there exist infinitely many $n$ such that $\log_p  d_p \G_n  \geq (\log_p [\G:\G_n])^\beta $, for some $\beta \in ]0,1[$, where $\G_n$ is the Zassenahaus filtration of $\G$.
 \end{theo}

 In our context, as corollary, one obtains:
 
 \begin{coro}
 Let $\K$  be a totally imaginary field of degree at least $12$ over $\Q$. Choose $p>2$ such that : $(i)$ $p$ splits totally in $\K/\Q$ and, $(ii)$ $\K$ is $p$-rational.
Then there exists a constant $\beta >0$ and a sequence of  $p$-rational number fields $(\L_n)$ in $\K_{S_p}/\K$ such that $$\log d_p \Cl_{\L_n} \gg (\log [\L_n:\Q])^\beta,$$
where $\Cl_{\L_n}$ is the class group of $\L_n$. 
 \end{coro}
 
 \begin{proof}Choose $k$ as in proof of Theorem \ref{theo:unramified-prational} such that $\K_{S_p}^{[k],p-ab}/\K$ is infinite. Put $\G=\Gal(\K_{S_p}^{[k],p-ab}/\K)$, and consider $\G_n$ the Zassenhaus filtration of $\G$. Let $\K_n$ be the subfield of $\K_{S_p}^{[k],p-ab}/\K $ fixed by $\G_n$: by Proposition \ref{prop:prationalheridity} all the fields $\K_n$  are $p$-rational. Take $n_0$  large enough so that  the pro-$p$ extension $\K_{S_p}^{[k],p-ab}/\K_{n_0}$ is unramified: this is always possible because $G_n$ forms a filtration of open subgroups of $\G$.

By hypothesis the group $\G$ passes the Golod-Shafarevich test: by Theorem \ref{theo:jaikin}, for infinitely many $n\geq n_0$,  we get  $\log_p  d_p \G_n  \geq (\log_p [\G:\G_n])^\beta $, for some $\beta \in ]0,1[$. To conclude it suffices to note that  for $n\geq n_0$ the extension
  $\K_{S_p}^{[k],p-ab}/\K_n$ is unramified, and then by class field theory one has $d_p \Cl_{\K_n} \geq d_p \G_n$.
 \end{proof}
 
 \begin{rema}
 When the class group of $\K$ is not trivial, $p$-rationality of $\K$ implies the Hilbert $p$-class field $\HH$ of $\K$ must be contained 
in the compositum of the $\Z_p$-extensions of $\K$,  something that can be difficult to arrange. See \cite[Chapter IV, \S 3]{gras} for a good explanation.
 \end{rema}


\section{Depth  of ramification}

\subsection{Motivation} 
Let us start with one comment that  motivates this section.
Let $P(t)=1-dt+rt^2$ be a polynomial  with no root on $[0,1]$ but such that 
$1-dt+(r-1)t^2$ has a root. For example, take the totally real field of $\S 3.3.2$ where $(d,r)=(9,21)$ .

\medskip

Suppose that $\G_\emptyset$ has $P$ as Golod-Shafarevich polynomial for a certain minimal presentation $(\PP)$. Then with $S=\{\p\}$, $\p$ coprime to $p$, and when $\mu_p \subset \K$, the group $\G_S$ has parameters $(d,r-1)$, or $(d+1,r)$. If $r>1$, it is easy to see  $\G_S$ is infinite in either case.
Suppose now that 
 the generator of  inertia  at $\p$, $\tau_\p$, has   depth at least $k$ in $\G_S$. If we cut $\G_S$ by $\langle \tau_\p \rangle$, the 
 Golod-Shafarevich polynomial becomes $1-dt+(r-1)t^2+t^{k}$, and for large $k$  it has a root.  In this case we can introduce the relation
 $\tau_\p$ and observe $\K_\emptyset/\K$ is infinite. 
 For  $(d,r)=(9,21)$ we need $k \geq 3$.

\begin{Question} Suppose $\G_S$ is infinite.
To simplify, take $S=\{\p\}$ with $\p$ coprime to~$p$. How deep can the generator $\tau_\p$  of tame inertia be in $\G_S$?
\end{Question}

\subsection{Add a splitting condition}
\subsubsection{Detect the level of inertia}
Let $S$ and $T$ be finite disjoint sets of primes of $\K$. We denote by 
$\K_S^T$ the maximal pro-$p$ subextension of $\K_S/\K$ where all places of $T$ split completely. Put $\G_S^T=\Gal(\K_S^T/\K)$. 
Consider the Frattini series $\Phi_n(\G_\emptyset^T)$ and $\Phi_n(\G_S^T)$. Set $\K_n:=(\K_\emptyset^T)^{\Phi_n(\G_\emptyset^T)}$.
We abuse notation and set $\omega:=\omega_{\G_S^T}$.

If $\L/\K$ is a finite subextension of $\K_S^T/\K$, we denote by $\G_{\L,S}^T$ the Galois group $\Gal(\L_S^T/\L)$.

\smallskip

To simplify, we assume that $S=\{\p\}$ where $\p \subset \O_\K$ is coprime to $p$. Let $\tau_\p \in \G_S^T$ be a generator of the inertia group at $\p$ in $\K_S^T/\K$.

\begin{lemm} \label{lemma:detectlevel}
If $d_p \G_{\K_n,\emptyset}^T=d_p \G_{\K_n,S}^T$ for some $n$, then $\omega(\tau_\p) \geq 2^{n}$.
\end{lemm}

\begin{proof}For $m\geq 1$, write  $\K_m'=(\K_S^T)^{\Phi_m(\G_S^T)}$. 
If, for $i \leq n+1$, we had a $p$-extension of $\K_i$ unramified outside $S$ but actually ramified there, we could take its composite with the unramified
extension $\K_n/\K_i$ 
to contradict the equality of our hypothesis. Thus
 $d_p \G_{\K_i,\emptyset}^T=d_p \G_{\K_i,S}^T$ for all   $i\leq n+1$. 
Hence, one has $\K_2= \K_2'$,  then $\K_3=\K_3'$ etc.  up to $\K_{n+1}=\K_{n+1}'$. In particular the extension $\K_{n+1}'/\K$ is unramified and  $\tau_\p \in \Phi_{n+1}(\G_S^{T})$; in other words $\G_S^T/\Phi_{n+1}(\G_S^{T}) \simeq \G_\emptyset^T/\Phi_{n+1}(\G_\emptyset^{T})$. By Proposition \ref{corollary:detect} $(iii)$, we get $\omega(\tau_\p) \geq 2^n$.
\end{proof}

\subsubsection{Depth and freeness}
 Recall from \S 3.2 that
  $$\V_S=\{x\in \K^{\times} | x \in \K^{\times p}_v \ \forall v \in S  \mbox{ and} \ v(x) \equiv 0 \mbox{ mod } p , \ \forall v\}$$ and
set  $$\V^T=\{x\in \K^{\times} | \ v(x)\equiv \ 0  \ {\rm mod}  \  p, \ \forall v \notin T\}$$ and
$$\V^T_S=\{x\in \K^{\times} |  x \in \K^{\times p}_v \, \forall \,v \in S \mbox{ and}   \ v(x)\equiv \ 0  \ {\rm mod}  \  p, \ \forall v \notin T  \}.$$

Note  $\V^\emptyset_S = \V_S$. If we switch fields to some $\L \supseteq \K$, we will include the field in the notation
to avoid confusion, e.g. $\V^T_\L$ or $\V^T_{\L,S}$. 
One  has 
\begin{eqnarray} \label{ES1} 1\longrightarrow \O_{\K,T}^{\times}/\O_{\K,T}^{\times p} \longrightarrow \V_\K^T/\K^{\times p} \longrightarrow \Cl_\K^T[p] \longrightarrow 1,\end{eqnarray}
where  $\O_{\K,T}^{ \times}$ denotes the group of $T$-units  of $\K$ and $\Cl_\K^T$ the $p$-Sylow subgroup of the $T$-class group of $\K$.
Put $\K':=\K(\mu_p)$ and
$\K_{(T)}'=\K'(\sqrt[p]{\V^T})$.
 One has  \cite[Chapter V, Corollary 2.4.2]{gras} involving the Artin symbol $ \FF{\K_{(T)}'/\K'}{\cdot}$ in $\Gal(\K_{(T)}'/\K')$.
 Note that $S$ and $T$ there are our $T$ and $S$ here respectively. 
\begin{theo}[Gras] \label{theo:Gras}
Let $S=\{\p_1,\cdots, \p_s\}$ be a tame set of  primes of $\K$.
There exists a cyclic degree $p$ extension $\F/\K$, unramified outside $S$, totally ramified at $S$, and where each place of $T$ splits completely, if and only if for  $i=1,\cdots, s$, there exist $a_i \in \fq_p^\times$ such that
$$\prod_{i=1}^s \FF{\K_{(T)}'/\K'}{\P_i}^{a_i} =1 \ \in \Gal(\K_{(T)}'/\K'),$$
where $\P_i|\p_i$ is any prime in  $\O_{\K'}$ above $\p_i$. 
\end{theo}

Choose now $\p$ of $\K$ whose Frobenius in $\K'_{(T)}/\K$ has order a multiple of $p$ and set $S=\{\p\}$.
 This implies that for $\P\subset \O_{\K'}$ with $\P | \p$ the Frobenius $x_\P \in \Gal(\K'_{(T)}/\K')$ at $\P$ is nontrivial.
Theorem~\ref{theo:Gras} implies  that $d_p \G_S^T=d_p \G_\emptyset^T$, hence $\tau_\p$ has depth at least 2 in $\G_S^T$ by Lemma~\ref{lemma:detectlevel}. 
We want to apply  this principle to the number field $\K_2=(\K_\emptyset^T)^{\Phi_2(\G_\emptyset^T)}$ bearing in mind the Galois action of $\G_\emptyset^T/\Phi_2(\G_\emptyset^T)$.

\medskip

Let us fix a Galois extension $\L/\K$ with Galois group $H$ inside $\K_\emptyset^T/\K$.
One has the following consequence of Theorem \ref{theo:Gras}.

\begin{coro} \label{lemm:freesubmodule} 
Suppose that $\V_{\L}^T/\L^{\times p}$ has a non trivial free $\fq_p[H]$-submodule.
Then there exists a prime $\p \subset \O_\K$ such that   $d_p \G_{\L,\emptyset}^T=d_p \G_{\L,S}^T$ where $S=\{\p\}$. 
\end{coro}

\begin{proof} 
Suppose that $\V_{\L}^T/\L^{\times p}$ has a  free $\fq_p[H]$-submodule $M$ of rank $1$; as the algebra $\fq_p[H]$ is Frobenius, the  free submodule $M$ is a  direct factor in $\V_{\L}^T/\L^{\times p}$. By Kummer duality, one deduces that $\Gal(\L'_{(T)}/\L')$ contains a free $\fq_p[H]$-module of rank $1$, generated by some $g$.  By Chebotarev's density theorem, choose a prime $\p \subset \O_\K$ such that its Frobenius in $\Gal(\L_{(T)}'/\K)$ is in the conjugacy class of $g$; the prime $\p$ splits completely in $\L'/\K$.  Put $S=\{\p\}$.
Denote by $\QQ_0$ a prime ideal in $\O_{\L'}$ above $\p$ such that $\FF{\L'_{(T)}/\L'}{\QQ_0}=g$. 
But $\forall h \in H$ we have
$\FF{\L'_{(T)}/\L'}{\QQ_0^h}=g^{h^{-1}}$, and as $\langle g \rangle_H$ is a free $\fq_p[H]$-module, there is no nontrivial relation between the $\FF{\L'_{(T)}/\L'}{\QQ_0^h}$'s. By Theorem \ref{theo:Gras}, there is no degree $p$ cyclic extension $\F/\L$ ramified at some places of $S$, unramified outside $S$, and where each place of $T$ splits completely, and then
$d_p \G_{\L,\emptyset}^T = d_p \G_{\L,S}^T$
\end{proof}

\begin{rema}
Let $\L/\K$ be a Galois extension of number fields with Galois group $H$. Let $T$ be a $H$-invariant set of places of $\L$. Here are  two ways to produce situations where  $\V_\L^T/\L^{\times p}$ has a free-$H$-part:
\begin{itemize}
\item[$(i)$] for large $|T|$ with size  depending on $|H|$ (thanks to a bound given by Ozaki \cite{Ozaki}, see also \cite{Hajir-Maire-analytic}), we are guaranteed that  $ \O_{\L,T}^{\times} \otimes \fq_p$, and then also $\V_\L^T/\L^{\times p}$ by (\ref{ES1}), contains a nontrivial free $\fq_p[H]$-submodule. 
But the bound for $|T|$ is very bad;
\item[$(ii)$] by  Kummer theory, and by an appropriate choice of $T$, we are guaranteed that   $\V_\L^{T}/{\L^\times}^p$ contains a nontrivial free $\fq_p[H]$-submodule. This is the
method we will use.
\end{itemize}
\end{rema}


\subsection{A result} \label{section:depth-inertia}

We  prove the following

\begin{theo} \label{theo:depth} Let $\K$ be a number field. 
 Then given $k>0$ there exists two different primes   $\q$ and $\p$, coprime to $p$,  such that 
$\omega(\tau_\p)\geq k$ in $\Gal(\K_{\{\p\}}^{\{\q\}}/\K)$. \end{theo}

As we will see the proof of Theorem \ref{theo:depth} can be reduced to the next Proposition.

\begin{prop} \label{theo:free-part}
Let $\L/\K$ be a given finite $p$-extension with Galois group $H$. 
There exists a positive density set $\Theta_1$ of primes $\q$ of $\K$, all coprime to  $p$ and all split  completely  in $\L/\K$, such that for all 
 $t\in \mathbb N$ 
and for all sets $T=\{\q_1,\cdots, \q_t \} \subset \Theta_1$ of $t$ different primes, one has  $$ \bigoplus_{i=1}^t \fq_p[H] \hookrightarrow \V_\L^T/\L^{\times p}.$$
\end{prop}

\begin{proof} Put $\L'=\L(\mu_p)$, and  $\Delta=\Gal(\L'/\L)$.  
Let us start with the following lemma 

\begin{lemm} \label{lemm:sans-p} 
There exists a positive density set $\Theta_1$ of primes $\q$ of $\K$, all coprime to  $p$, such that for all $t \in \Z_{\geq 0}$  and for all sets
 $T=\{\q_1,\cdots, \q_t \} \subset \Theta_1$ of $t$ different primes, one has  $ \fq_p\otimes \G_{\L',T}^{ab} \twoheadrightarrow  \bigoplus_{i=1}^t \fq_p[H][\Delta]$ as $\Gal(\L'/\K)$-modules.
Moreover the $\q_i$'s split completely in $\L/\K$.
\end{lemm}

\begin{proof}   First, let us choose  a set of places $\Sigma$ of $\L'$, $\Gal(\L'/\K)$-invariant,  such that $G_{\L',\emptyset}^{\Sigma,ab}=\{1\}$.
Put $\F=\L'(\sqrt[p]{\O_{\L',\Sigma}^\times})$.  The extension $\F/\K$ is Galois and  let $\Theta_1$ be the 
 Chebotarev   
set of places of  $\K$ that split completely in $\F/\K$. The splitting implies that for  each $v \in \Theta_1$ one has 
 the equality of completions
$\K_v=\L'_v=\L'_v(\sqrt[p]{\O_{\L',\Sigma}^\times})$, so  $\O_{\L',\Sigma}^{\times} \subset U_v^p$ and $\mu_p \subset \K_v$.
Take now   $T=\{v_1,\cdots, v_t\}$ a set of $t$ different places of $\Theta_1$  all coprime to~$p$. By class field theory one 
has: $$\fq_p\otimes \Gg_{\L',T}^{\Sigma,ab} \simeq \dfrac{\prod_{w|v\in T} U_w}{\O_{\L',\Sigma}^{\times} \prod_{w|v\in T} U_w^p} \simeq \prod_{w|v\in T} U_w/U_w^p.$$
As all places of  $T$ split completely  in $\L'/\K$, then $$\prod_{w|v \in T} U_w/U_w^p \simeq \fq_p[H][\Delta].$$ 
One concludes by noting that  $\fq_p\otimes \G_{\L',T}^{ab} \twoheadrightarrow \fq_p\otimes \Gg_{\L',T}^{\Sigma,ab}\simeq \bigoplus_{i=1}^t \fq_p[H][\Delta]$.
\end{proof}

The Kummer radical $R_T$ of the maximal $p$-elementary extension of~$\L'$, unramified outside~$T$, is a subgroup of $\V_{\L'}^T/({\L'}^\times)^p$ (remember that $T$ is coprime to $p$). Hence by Lemma~\ref{lemm:sans-p} we get
$$\bigoplus_{i=1}^t \big(\fq_p[H][\Delta]\big)^*  \hookrightarrow R_T \subset \V_{\L'}^T/({\L'}^\times)^p,$$
where $*$ denotes the reflection action following $\Delta$. By taking the $\Delta$-invariant, 
Proposition~\ref{theo:free-part} holds by noting that $\left(\V_{\L'}^T/({\L'}^\times)^p\right)^\Delta \simeq \V_{\L}^T/{\L}^{\times p}$.
\end{proof}

\begin{proof}(of Theorem \ref{theo:depth}).
 Given $k>0$, write $n=\lceil \log_2k \rceil$.
 Let $\L=(\K_\emptyset)^{\Phi_n(\G_\emptyset)}$ and put $H=\Gal(\L/\K)$. By Proposition \ref{theo:free-part}, choose a prime $\q$, coprime to $p$ and that splits completely in $\L/\K$, such that $\fq_p[H] \hookrightarrow \V_{\L}^T/\L^{\times p}$ where $T=\{\q\}$. 
So $\V_\L^T/\L^{\times p}$ contains a free nontrivial $\fq_p[H]$-submodule.  Then by Corollary~\ref{lemm:freesubmodule}, there exists a prime $\p \subset \O_\K$  such that  $d_p \G_{\L,\emptyset}^T=d_p \G_{\L,S}^T$ where $S=\{\p\}$.

But, as $\q$ splits completely in $\/\K$, observe now that  $(\K_\emptyset)^{\Phi_n(\G_\emptyset)}=(\K_\emptyset^T)^{\Phi_n(\G_\emptyset^T)}$.
Then by Lemma \ref{lemma:detectlevel}, we get that the depth of $\tau_\p \in \G_S^T=\Gal(\K_S^T/\K)$ is at least $2^n \geq k$.
\end{proof}

\begin{rema} The reader may wonder why one can't, for instance in the totally real example of \S 3.3.2, simply apply Theorem~\ref{theo:depth}
for some $\p$ whose $\tau_\p$ will have depth $3$ for some $\q$. The difficulties are that first $\G^{ \{\q\} }_{ \{\p\} }$ may have many more
relations imposed by the splitting condition, and second if one removes the splitting condition,  the kernel of the map
$\G_{ \{\p\} }    \twoheadrightarrow \G^{ \{\q\} }_{ \{\p\} }$ might contain elements of depth $2$, so  the preimage of $\tau_\p$ 
may have depth $2$.
\end{rema}



\end{document}